\documentclass{article}
\usepackage{maa-monthly}
\def\bc{\begin{center}}
\def\ec{\end{center}}
\def\ig{\includegraphics}
\theoremstyle{theorem}

\newtheorem{proposition}{Proposition}

\theoremstyle{definition}
\newtheorem*{definition*}{Definition}
\newtheorem*{remark*}{Remark}
\newtheorem*{lemma*}{Lemma}
\def\R{\mathbb{R}}

\definecolor{darkgreen}{rgb}{0.0,0.6,0.0}
\begin{document}

\title{A blue sky bifurcation in the dynamics of political candidates}
\markright{Dynamics of political candidates} 
\author{Christoph B\"orgers, Bruce Boghosian, Natasa Dragovic, and \\ Anna Haensch} 

\maketitle

\begin{abstract} Political candidates often shift their positions opportunistically in hopes of capturing more votes.
When there are only two candidates, 
the best strategy for each of them is often to move towards the other. This eventually results
 in  two centrists with coalescing views.  However, the strategy of 
moving towards the other candidate ceases to be optimal when enough voters abstain instead of
voting for a centrist who does not represent their views.
These observations, formalized in  various ways, 
have been made many times. Our own formalization is based on 
differential equations. The surprise and main result derived
from these equations is that the final candidate positions
can jump discontinuously 
as the voters' loyalty towards their candidate wanes.
The underlying mathematical mechanism is a blue sky bifurcation.
\end{abstract}

\section{Introduction.}

Among U.S.\ Presidents of recent decades, Bill Clinton is a prime example of a centrist \cite{Velasco_book}.
Especially after the 1994 midterm elections, which were disastrous for his party,
Clinton's positions moved to the
right. He easily won reelection in 1996. 
Donald Trump, on the other hand, exemplifies the antithesis of a centrist. He took positions during the 2016 campaign that were so far to the right that many 
thought he could not possibly win.  He stuck with those positions, and was elected.

We study a mathematical model, a system of differential equations, that may help understand phenomena of this sort. In constructing our equations we aim at simplicity, and  leave out many aspects of reality. We assume that voter and candidate
positions can be described by a single number, their position on a ``left-right axis," that the views of voters can be described by a simple density, and so on. 
These assumptions are common in the mathematical literature on opinion dynamics.
We believe that mathematicians can contribute 
to the social sciences by analyzing such highly idealized models and deriving insights from that analysis, which may
then be tested in the real world, or even in much more detailed and realistic computer simulations. 

As voters become more inclined to abstain rather than vote for a centrist who does not represent their views, centrism 
will eventually cease to be a politician's best strategy. This point has been made by many previous authors,  formalized in various
different 
ways;  
see for instance \cite{callander_wilson, siegenfeld_bar-yam,Leppel}.   
In our model, the transition sometimes involves a blue sky bifurcation 
\cite{strogatz_book} resulting in a discontinuous jump in optimal candidate positions.
Our analysis indicates that such discontinuities appear when the electorate is polarized, and the
candidates are markedly different in their willingness to shift positions opportunistically.

We know of course that small shifts in the electorate can have large consequences for the {\em outcome} of a tight election. 
What we show here is that small shifts in the electorate can also have large consequences for the {\em optimal strategy} of a political candidate.

\section{Distribution of views among voters.}

\subsection{Density of voter positions.}
We make the simplifying assumption that any individual's political views can be characterized by where they stand on a ``left-right
axis." 
So an individual's political views are characterized by a real number $x$. When that number is  in the
negative range, we call them ``left-wing," and when it is  in the positive range,  ``right-wing." 
We assume that the distribution of views among voters can be characterized by a density $f$,  
so for $a<b$, 
$$
\int_a^b f(x) ~\! dx
$$
represents the fraction of voters whose views fall between $a$ and $b$. Throughout
the paper, we assume $f$ to be ``nice" in the following sense.

\begin{enumerate}
\item $f$ is continuously differentiable, bounded and with a bounded derivative,
\item $f(x)>0$ for all $x \in \R$.
\end{enumerate}
The location on the $x$-axis  labeled  ``$0$" and considered ``the center" is arbitrary. 
We simply make the convention that $0$ is the median voter position, so we make  a third assumption
on $f$: 
\begin{enumerate}
\item[3.] 
$
\int_{-\infty}^0 f(x) dx = \frac{1}{2}.
$
\end{enumerate}
Assumptions 1--3 underlie our discussion in various places, and we will not always state them explicitly from here on.

\subsection{Unimodal and bimodal densities.}
All numerical examples in this paper use one of two representative densities,
a {\em unimodal} one defined by
\begin{equation}
\label{unimodal}
f(x) = \frac{e^{-x^2/2}}{\sqrt{2 \pi}},
\end{equation}
and a {\em bimodal} or {\em polarized} one \cite{Pew} given by
\begin{equation}
\label{bimodal}
f(x) = \frac{e^{-2(x-1)^2} +  e^{-2(x+1)^2} }{\sqrt{2 \pi} }.
\end{equation}
See Figure \ref{fig:PDFCDF}.
\begin{figure}[h!]
\bc
\ig[scale=0.5]{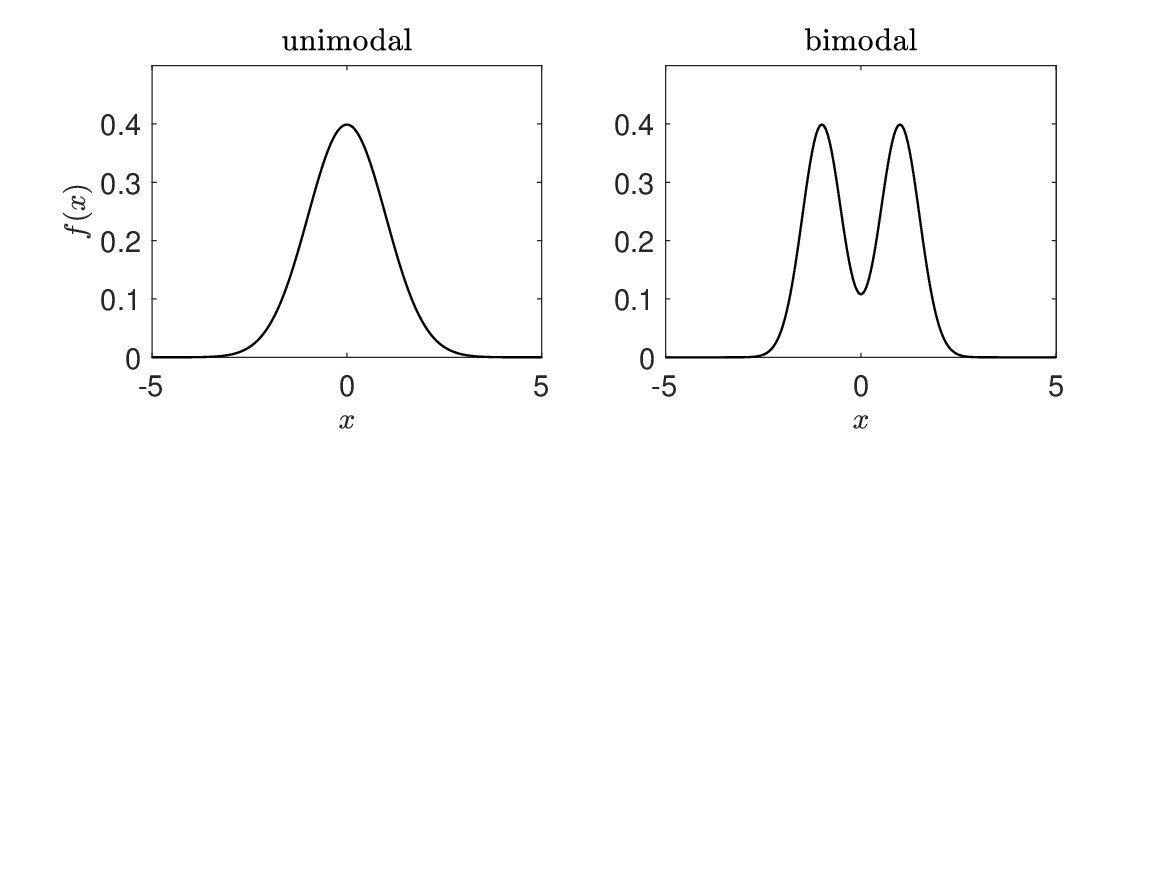}
\caption{Unimodal and bimodal distributions of views among voters.}
\label{fig:PDFCDF}
\ec
\end{figure}

\pagebreak

\section{The well-known benefits of centrism.}
\label{sec:benefits_of_centrism}

\subsection{The median voter theorem.}
We consider two candidates called $L$ and $R$. Their positions on the $x$-axis are
$\ell$ and $r$, with $\ell < r$. We assume that a voter whose position is closer to $\ell$ than to $r$ will vote for $L$, 
and a voter whose position is closer to $r$ than to $\ell$ will vote for $R$.  
We also assume, for now, that everybody votes. The fractions of voters who vote for $L$ and $R$, respectively,  are
then

\begin{equation}
\label{SLSR}
S_L = \int_{-\infty}^{(\ell+r)/2} f(x) dx ~~~\mbox{and}~~~ S_R =  \int_{(\ell+r)/2}^\infty f(x) dx.
\end{equation}

\noindent
\begin{proposition}  \label{prop:whowins} $L$ wins if $(\ell+r)/2>0$, and $R$ wins if $(\ell+r)/2<0$. Equivalently, the 
candidate whose position is closer to $0$ wins.
\end{proposition} 
\begin{proof}

Since by definition, $0$ is the median of the voter distribution, $S_L>1/2$ if and only if $(\ell+r)/2>0$, and $S_R>1/2$ if and
only if $(\ell+r)/2<0$. Since $r-\ell>0$, 
$$
(\ell+r)/2> 0 ~~\Leftrightarrow ~~
(r+\ell)(r-\ell)>0 ~~\Leftrightarrow ~~ r^2-\ell^2>0 ~~\Leftrightarrow~~ |r|>|\ell|.
$$
Similarly, 
$
(\ell + r)/2 < 0 ~~\Leftrightarrow~ |r| < |\ell|.
$
\end{proof}

Proposition \ref{prop:whowins} is a special case of the {\em median voter theorem} \cite{Duncan_Black}.

\subsection{Hotelling's law: Coalescence of candidate views.}

Proposition \ref{prop:whowins} tells us that the (globally) optimal candidate position is $0$, the median. 
Here we think about how candidates might move gradually to improve their positions.

If both candidates are opportunistic, they might for instance follow a system of ordinary differential equations
such as 
\begin{eqnarray}
\label{dell}
\frac{d \ell}{dt} &=& \alpha \frac{\partial S_L}{\partial \ell} (\ell,r) = ~~~ \frac{\alpha}{2} f \left( \frac{\ell+r}{2} \right), \\
\label{dr}
\frac{d r}{dt} &=& \beta \frac{\partial S_R}{\partial r} (\ell,r) = - \frac{\beta}{2} f \left( \frac{\ell+r}{2} \right), 
\end{eqnarray}
where $\alpha$ and $\beta$ are positive constants.
The equations reflect the assumption that $L$ and $R$ optimize their share using a continuous steepest-ascent 
approach. That is, they always move in the beneficial direction, and they move faster when moving is more beneficial. 
This is a {\em local} optimization procedure.

The constants $\alpha$ and $\beta$ measure the  eagerness with which $L$ and $R$ move opportunistically; if $\alpha>\beta$,
then $L$ is the more eager opportunist. Conversely, if $\beta>\alpha$, $R$ is the more eager opportunist.

Because we assume that $f$ is everywhere strictly 
positive, the graphs of $\ell$ and $r$ meet each other in finite time, transversally, as illustrated in Figure \ref{fig:MEET}.
When this happens, the assumption $\ell < r$ underlying our definitions of $S_L$ and $S_R$ ceases to hold, 
and we assume that $\ell$ and $r$ stop moving at that moment. 

In Figure \ref{fig:MEET}, the functions $\ell(t)$ and $r(t)$ are almost, but not quite, linear. In fact, equations (\ref{dell}) and (\ref{dr}) show that
$\ell$ and $r$ are linear functions of $t$ if and only if $(\ell+r)/2$ is constant, and this in turn is the case if and only if $\alpha=\beta$ (average equations (\ref{dell}) and 
(\ref{dr}) to see this). We remark parenthetically that $\beta \ell + \alpha r$ is always constant (multiply 
 (\ref{dell}) by $\beta$ and (\ref{dr}) by $\alpha$, then sum, to see this), so $\ell$ and $r$ are always linear functions {\em of each other}, assuming
 neither $\alpha$ nor $\beta$ is zero.

{ The coalescence of candidate views demonstrated here is an example of {\em Hotelling's Law} \cite{hotelling}, 
which holds --- in the economic context --- that it is often best for producers to make their products as similar to each other as possible.  Hotelling
\cite{hotelling} mentions Democrats and Republicans as an example, observing that ``each party strives to make its platform as much like the other's as possible." 
Another amusing example from Hotelling's paper: ``Methodist and Presbyterian churches are too much alike."
}
\begin{figure}
\bc
\ig[scale=0.4]{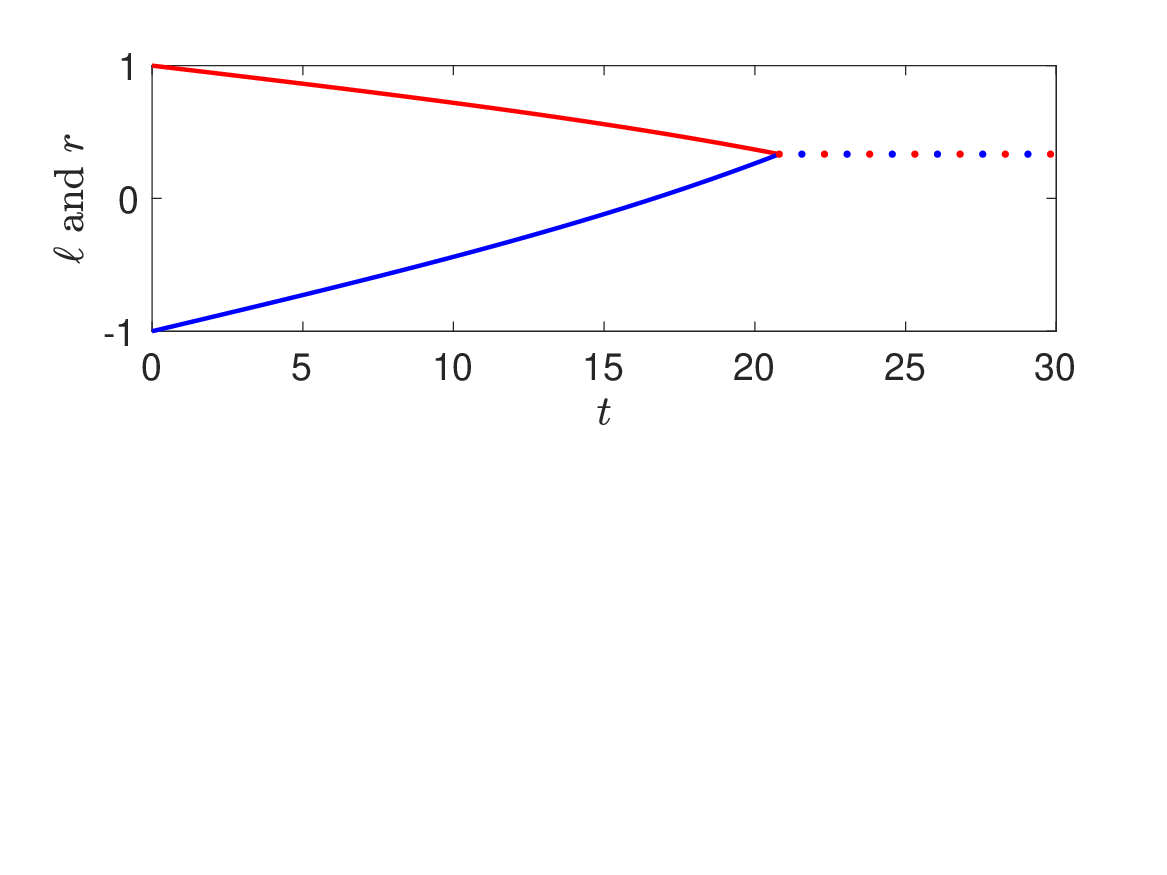}
\caption{$\ell$ and $r$ moving according to equations (\ref{dell}) and (\ref{dr}), with  $\alpha=1$ and $\beta=0.5$.}
\label{fig:MEET}
\ec
\end{figure}
\subsection{Different degrees of opportunism.}
In the example of Figure \ref{fig:MEET},\footnote{All figures in this paper were generated with Matlab code  available by e-mail from the first author.
See \url{https://annahaensch.com/centrism.html} for a 
 Python workbook with code generating similar figures.}
$L$ is the
more eager opportunist: $\alpha = 1> \beta=0.5$. As a result, $(\ell + r)/2$ is an increasing function of time (average
equations (\ref{dell}) and (\ref{dr}) to see this), and  $\ell$ and $r$ meet  to the right of $(\ell(0)+r(0))/2$. The more eager opportunist 
gains more{; this point was also made in \cite{Leppel}.} (Who wins still depends on  initial conditions. For instance, if $r<0$ initially,
then $L$ can never win, no matter how eagerly $L$ moves to the right, since $\ell$ cannot move past $r$.) However,  the less eager opportunist  has greater impact on the position adopted by the other
candidate. 

This may be illustrated by the  April 2022 presidential election in France. Emmanuel Macron
moved towards the right while his leading rival on the right, Marine LePen, moved towards the center. Macron won, but was
also arguably the one whose positions underwent the greater shift \cite{NYTimes}. 

Two months later, in the 
National Assembly elections, Macron's Ensemble coalition fell far short of an absolute majority, losing votes to the
 left-wing NUPES coalition as well as  to the right-wing Rassemblement National. Abstentions were
 at a record number. This leads us to the following modification of our model.

\section{Declining voter loyalty can make centrism sub-optimal.}

\subsection{Voter loyalty.}
Up to now we have assumed that a voter always votes for the candidate whose views are closest to theirs. We use the expression {\em voter loyalty} 
to denote 
a voter's willingness to support the candidate closest in views to them, even if that candidate's views are still far from their own. 
In the political science literature, the  expression {\em voter alienation} is common; it denotes the absence of voter loyalty. 

A decrease in loyalty (increase
in alienation) means an increase in the voters' willingness to abstain when no candidate is close to their views.
This can prevent Hotelling's Law from taking hold. As we have mentioned, this point has been made many times. The idea was actually observed by Hotelling himself in 1929 \cite{hotelling}:

\begin{quote} 
``The increment in B's sales to his more remote customers when he moves nearer them may be more than compensation to him for abandoning some of his nearer business to A. In this case B will definitely and apart from extraneous circumstances choose a location at some distance from A."
\end{quote}

\noindent
We describe the mechanism in the context of our model here.

\subsection{Modeling voter abstentions.} 

If a voter's position is $x$, and $x$ is closer to $\ell$ than to $r$, then the voter will, in our model as described
so far, vote for $L$ rather than
for $R$. However, if $|\ell-x|$ is large, the voter might just stay home and
not vote at all. The share of the vote that $L$ receives might therefore become

\begin{equation}
\label{eq:tired}
S_L = S_L(\ell,r) = 
\int_{-\infty}^{(\ell+r)/2} f(x) g(|\ell-x|) ~\! dx
\end{equation}
where 
$
g: ~ [0,\infty) \rightarrow (0,1]
$
is a decreasing function with $g(0)=1$. Among the voters at $x< (\ell+r)/2$, only a fraction $g(|\ell-x|)$  vote at all. Those who do, of course, still vote for $L$, since $x$ is closer to $\ell$ than to $r$ when $x< (\ell+r)/2$.
Similarly we assume
\begin{equation}
\label{eq:tiredR}
S_R =S_R(\ell,r) =  \int_{(\ell+r)/2}^\infty f(x) g(|r-x|) ~\! dx.
\end{equation}
It might make sense to make the function $g$ in (\ref{eq:tiredR}) different from that in (\ref{eq:tired}), since right-wing voters might
exhibit more, or less, loyalty than left-wing voters. However, to keep things as simple
as possible, we won't do that here. 

We will assume $g$ to be ``nice" in the following sense.
\begin{enumerate}
\item $g$ is of the form 
$$
g(z) =g_0 \left( \frac{z}{\gamma} \right), ~~~z \geq 0,
$$
where $g_0$ is a fixed continuous function, and $\gamma>0$ a parameter,
\item $g_0(z)$, $z \geq 0$, is decreasing, with $g_0(0)=1$, 
\item $g_0(z) \leq O  (1/z^2)$ as $z \rightarrow \infty$, 
\item $\int_0^\infty g_0(z) dz= 1$.
\end{enumerate}
The third assumption implies that $\int_0^\infty g_0(z) dz< \infty$. The fourth assumption, which
will simplify our notation a little bit later on, is then just 
a matter of normalization. If $\int_0^\infty g_0(z)  dz = C$ with $0 < C < \infty$, then $\int_0^\infty g_0(Cz) ~\! dz = 1$, so 
we could simply replace $g_0(z)$ by $g_0(Cz)$ to satisfy our fourth assumption.
We will use these assumptions from here on without always stating them explicitly. 

We think of $\gamma$ 
as a measure of voter loyalty. Greater $\gamma$ means 
greater loyalty. For a distance between the candidate and voter positions to cause
a large amount of voter disloyalty,  it must at least be on the order of $\gamma$. 
In our numerical examples, we will always use 
$$
g_0(z) = e^{-z}.
$$

\subsection{Candidate dynamics with voter abstentions.} As in Section \ref{sec:benefits_of_centrism}, we assume that $r$ and $\ell$ are functions of time, following the equations
\begin{eqnarray}
\label{eq:ODE1}
\frac{d \ell}{dt} &=& \alpha \frac{\partial S_L}{\partial \ell}(\ell,r), \\
\label{eq:ODE2}
\frac{dr}{dt} &=& \beta \frac{\partial S_R}{\partial r}(\ell, r).
\end{eqnarray}
(Compare equations (\ref{dell}) and (\ref{dr}).) We work out the partial derivatives on the right-hand side explicitly. 

\begin{lemma*} {\em Assume $S_L$ and $S_R$ to be defined as in equations (\ref{eq:tired}) and (\ref{eq:tiredR}). Then
for all $\ell$ and $r$ with $\ell<r$, 
$$
\frac{\partial S_L}{\partial \ell}(\ell,r) = - \frac{1}{2} f \left( \frac{\ell+r}{2} \right) g \left( \frac{r-\ell}{2} \right) 
\hskip 120pt
$$
\begin{equation}
\label{eq:partial_S_L_partial_ell}
\hskip 100pt
+ ~
\int_{-\infty}^{(\ell+r)/2} f'(x) g(|\ell-x|) dx
\end{equation}

\begin{equation}
\label{eq:partial_S_L_partial_ell0}
 = - \frac{1}{2} f \left( \frac{\ell+r}{2} \right) g_0 \left( \frac{r-\ell}{2 \gamma} \right)  + ~
\gamma \int_{-(r-\ell)/(2 \gamma)}^\infty f'(\ell-\gamma u) g_0(|u|) du
\end{equation}
and 
$$
\frac{\partial S_R}{\partial r}(\ell,r) =  \frac{1}{2} f \left( \frac{\ell+r}{2} \right) g \left( \frac{r-\ell}{2} \right) 
\hskip 120pt
$$
\begin{equation}
\label{eq:partial_S_R_partial_r}
\hskip 100pt
+  ~
\int_{(\ell+r)/2}^\infty f'(x) g(|r-x|) dx
\end{equation}
\begin{equation}
\label{eq:partial_S_R_partial_r0}
= \frac{1}{2} f \left( \frac{\ell+r}{2} \right) g_0 \left( \frac{r-\ell}{2 \gamma} \right)  
+  ~
\gamma \int_{-(r-\ell)/(2 \gamma)}^\infty f'(r+\gamma u) g_0(|u|) du.
\end{equation}
}
\end{lemma*} 
\begin{proof} By definition,
\begin{eqnarray}
\nonumber
S_L(\ell,r) &=& \int_{-\infty}^{(\ell+r)/2} f(x) g(|\ell-x|) dx.
\end{eqnarray}

\noindent
Therefore for $h>0$ with $\ell+h <  r$:
\begin{eqnarray*} 
&~& 
\frac{S_L(\ell+h,r) - S_L(\ell,r)}{h} =   \\
&~&
 \frac{1}{h} \left( \int_{-\infty}^{(\ell+h+r)/2} f(x) g(|\ell+h-x|) dx - 
\int_{-\infty}^{(\ell+r)/2} f(x) g(|\ell-x|) dx \right) =  \\
&~&
 \frac{1}{h} \left( \int_{-\infty}^{(\ell-h+r)/2} f(x+h) g(|\ell-x|) dx - 
\int_{-\infty}^{(\ell+r)/2} f(x) g(|\ell-x|) dx \right) =  \\
&~&
 \int_{-\infty}^{(\ell+r)/2} \frac{f(x+h)-f(x)}{h} g(|\ell-x|) dx - \frac{1}{h} \int_{(\ell-h+r)/2}^{(\ell+r)/2} f(x+h) g(|\ell-x|) dx.
 \end{eqnarray*} 
 Using the dominated convergence theorem, remembering  that $f'$ is bounded and $g$ is integrable, and 
 also using that both $f$ and $g$
 are 
 continuous, 
 (\ref{eq:partial_S_L_partial_ell}) now follows by taking the limit as $h \rightarrow 0$. Then (\ref{eq:partial_S_L_partial_ell0})
 follows with the substitution $ \frac{\ell-x}{\gamma}=u$. 
 
 The derivations of (\ref{eq:partial_S_R_partial_r}) and 
 (\ref{eq:partial_S_R_partial_r0}) are analogous.
\end{proof}

\subsection{As voter loyalty decreases, centrism becomes non-optimal eventually.}
We now prove that for sufficiently small $\gamma$, that is, for sufficiently disloyal voters, it will not be 
in $L$'s best interest to come arbitrarily close to $R$.

\begin{proposition}
\label{prop:R_repels}
For sufficiently small $\gamma$, 
$$
\frac{\partial S_L}{\partial \ell}(r,r) < 0.
$$
\end{proposition}

\begin{proof}
Set $\ell=r$ in equation (\ref{eq:partial_S_L_partial_ell0}) (that is, take the limit of $\frac{\partial S_L}{\partial \ell}(\ell,r)$ as $\ell \rightarrow r$): 
$$
\frac{\partial S_L}{\partial \ell}(r, r) = - \frac{f(r)}{2}  + \gamma \int_{0}^\infty f'(r- \gamma u) g_0(u) ~\! du.
$$
By the dominated convergence theorem, using the continuity of $f'$ and the assumption that $\int_0^\infty g_0(z) dz = 1$, 
$$
\lim_{\gamma \rightarrow 0} 
\int_{0}^\infty f'(r-\gamma u) g_0(u) du  = f'(r).
$$
Therefore 
$\hskip 50pt
\displaystyle{\frac{\partial S_L}{\partial \ell}(r, r) = - \frac{f(r)}{2} + O(\gamma) < 0 ~~~ \mbox{as $\gamma \rightarrow 0$}.}
$
\end{proof}
\begin{figure}
\bc
\ig[scale=0.5]{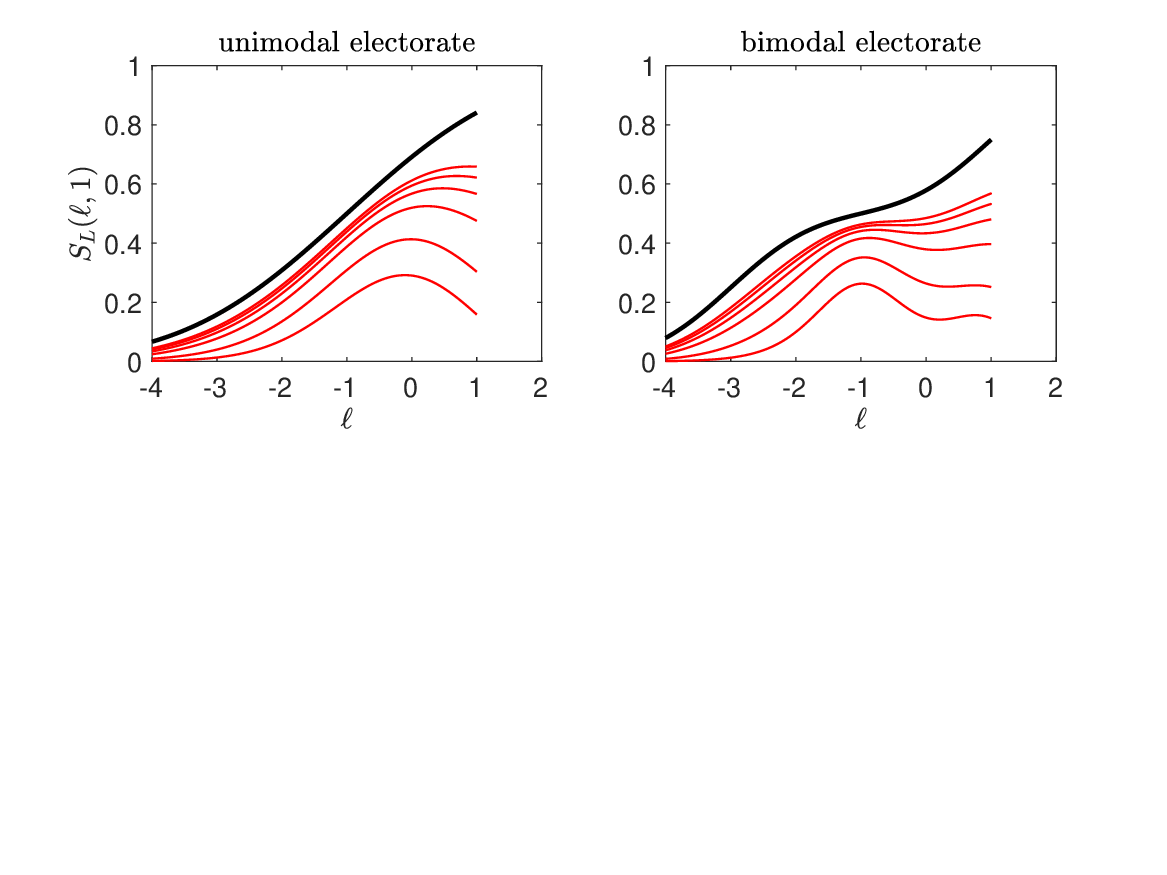}
\caption{$S_L(\ell,{ r=1})$ as a function of $\ell$, with the parameter $\gamma$  set to $\infty$ (black), $5$, $4$, 3, 2, 1, and 0.5. For fixed $\ell$, $S_L(\ell)$ decreases as $\gamma$ decreases.}
\label{fig:SL_WITH_GAMMA}
\ec
\end{figure}

\subsection{When voter loyalty is low, it's best to be positioned where there are many voters.}

This, too, is not a surprise of course, but we will verify here that our model predicts it. 
As an example, we plot $S_L(\ell,1)$ as a function of $\ell<1$
in 
Figure \ref{fig:SL_WITH_GAMMA}.  As $\gamma$ decreases, $S_L$ decreases, and
eventually loses its monotonicity as a function of $\ell$.
It turns out that in fact, in the limit as $\gamma \rightarrow 0$, the graph of $S_L(\ell,r)$ always mimics that
of $f$, up to scaling. 
(A hint of that is visible in Figure \ref{fig:SL_WITH_GAMMA}.) 
This is the content of the following proposition.

\begin{proposition}
\label{prop:SL_shadows_f} Let  $\ell$ and $r$ with $\ell<r$ be fixed. As $\gamma \rightarrow 0$,
\begin{equation}
\label{eq:S_L_asymptotics}
S_L(\ell, r) \sim 2 \gamma f(\ell) 
\end{equation}
and 
\begin{equation}
\label{eq:S_L_prime_asymptotics}
\frac{\partial S_L}{\partial \ell}(\ell, r)  \sim 2 \gamma  ~\! f'(\ell).
\end{equation}
\end{proposition}

\begin{proof} Using $g(z) = g_0(z/\gamma)$ and the substitution $(\ell-x)/\gamma=u$, we obtain
\begin{eqnarray}
\nonumber
S_L(\ell,r) &=& \int_{-\infty}^{(\ell+r)/2} f(x) g(|\ell-x|) dx 
\\ \label{stop} &=&  \gamma \int_{-(r-\ell)/(2 \gamma)}^\infty f(\ell-u \gamma) g_0(|u|) du.
\end{eqnarray}
Using the dominated convergence theorem and remembering that $f$ is bounded and continuous and $g_0$ is integrable 
with $\int_0^\infty g_0(z) dz = 1$ and therefore $\int_{-\infty}^\infty g_0(|u|) du =2$, we see that the integral in (\ref{stop}) converges to $2  f(\ell)$ as $\gamma \rightarrow 0$.
This implies (\ref{eq:S_L_asymptotics}). 

The proof of (\ref{eq:S_L_prime_asymptotics}) is based on
(\ref{eq:partial_S_L_partial_ell0}).
Since $g_0(z) \leq O(1/z^2)$ as $z \rightarrow \infty$, we have 
$$
\frac{1}{2} f \left( \frac{\ell+r}{2} \right) g_0 \left( \frac{r-\ell}{2 \gamma} \right)  \leq O(\gamma^2) ~~~\mbox{as $\gamma \rightarrow 0$}.
$$
Using the dominated convergence theorem and remembering that $f'$ is bounded and continuous, and $g_0$ is integrable,
$$
\lim_{\gamma \rightarrow 0} 
 \int_{-(r-\ell)/(2 \gamma)}^\infty f'(\ell-\gamma u) g_0(|u|) du = f'(\ell) \int_{-\infty}^\infty g_0(|u|) du = 2 f'(\ell).
 $$
This implies (\ref{eq:S_L_prime_asymptotics}). 
 \end{proof}

If $L$ starts out on the left, and opportunistically changes positions following
equation (\ref{eq:ODE1}), regions where $\frac{\partial S_L}{\partial \ell}<0$ will prevent $L$ from moving  right. This confirms again that the candidate positions won't come together when $\gamma$ is small enough.

A subtlety may be worth pointing out here. By Proposition \ref{prop:R_repels}, it is {\em always} true that for $\ell$ near $r$, $\frac{\partial S_L}{\partial \ell}(\ell, r)$ will
be negative when $\gamma$ is small. From  (\ref{eq:S_L_prime_asymptotics}), you might 
be tempted to conclude that that can't be 
true if $f'(r)>0$. 
That conclusion, however, is false because the
convergence of $\frac{1}{\gamma} \frac{\partial S_L}{\partial \ell}(\ell,r)$ to $2 f'(\ell)$  is not uniform in $\ell$.

\section{The shift in optimal candidate positions can be abrupt.} ~

\subsection{Blue sky bifurcations.}

When the continuous variation of a parameter in a system of ordinary differential equations causes 
a saddle and a stable node to collide, resulting in the annihilation of both equilibria, we say that the system undergoes a 
{\em saddle-node bifurcation.} When the parameter is varied in the reverse direction, two fixed points --- a stable node
and a saddle --- are suddenly created,  they appear ``out of the blue." Viewed in this direction, the bifurcation is therefore poetically called a  {\em blue sky 
bifurcation} \cite{abraham_shaw}.
We refer the reader to \cite{strogatz_book} for an excellent discussion of blue sky bifurcations.

\subsection{Opportunistic left-winger,  ideologically fixed right-winger.} 
Assume $\beta=0$,  
so $R$ is not at all opportunistic, and $r=1$ is fixed. 
The left-wing candidate moves according to 
\begin{equation}
\label{eq:one_ODE}
\frac{d \ell}{dt} =  \alpha \frac{\partial S_L}{\partial \ell}(\ell,1).
\end{equation}
This is a natural situation 
to think about. In fact, even when $r$ is {\em not} fixed, it seems reasonable to 
expect that $L$ will operate {\em as though} it were; otherwise $L$ would have to think around corners,
trying to anticipate $R$'s opportunistic moves while contemplating their own.

For various values of $\gamma$, we showed the right-hand side of equation (\ref{eq:one_ODE}), up to the
factor $\alpha$,  in Figure \ref{fig:SL_WITH_GAMMA} already.
There is a significant difference between the two panels of Figure 3. In the left panel, for the unimodal electorate,
negative slopes first appear in the vicinity of $r=1$ as $\gamma$ is lowered, and the (unique) maximum of $S_L$ gradually 
moves towards $0$, starting at $r=1$. In the right panel, for the bimodal electorate, negative slopes first appear much further to the  left, and a local maximum first appears near $\ell=-1$.

Figure \ref{fig:SL_WITH_GAMMA_DERIVATIVES} shows $\partial S_L/\partial \ell$ as a function of $\ell < r=1$, for various different values of $\gamma$. 
{ In the left panel, 
the first negative values of $\partial S_L/\partial \ell$, as $\gamma$ decreases, occur near $\ell=r=1$ (the fixed position of the right-wing candidate). A fixed point of the
equation $\frac{d \ell}{dt} =  \frac{\partial S_L}{\partial \ell}$ is created at $\ell=r=1$ and gradually moves left.}

 In the 
right panel of Figure \ref{fig:SL_WITH_GAMMA_DERIVATIVES}, as 
$\gamma$ decreases, the graph of $\partial S_L/\partial \ell$ pushes through the $\ell$-axis, creating a pair
of fixed points of the differential equation $\frac{d \ell}{dt} =  \frac{\partial S_L}{\partial \ell}$  in a blue sky bifurcation that occurs at some critical value $\gamma=\gamma_c$, in our example slightly above 4. We show a close-up of the right panel of Figure \ref{fig:SL_WITH_GAMMA_DERIVATIVES} again in Figure \ref{fig:SL_WITH_GAMMA_DERIVATIVES_CLOSEUP}, with the pair
of fixed points indicated. 
\begin{figure}[ht]
\bc
\ig[scale=0.5]{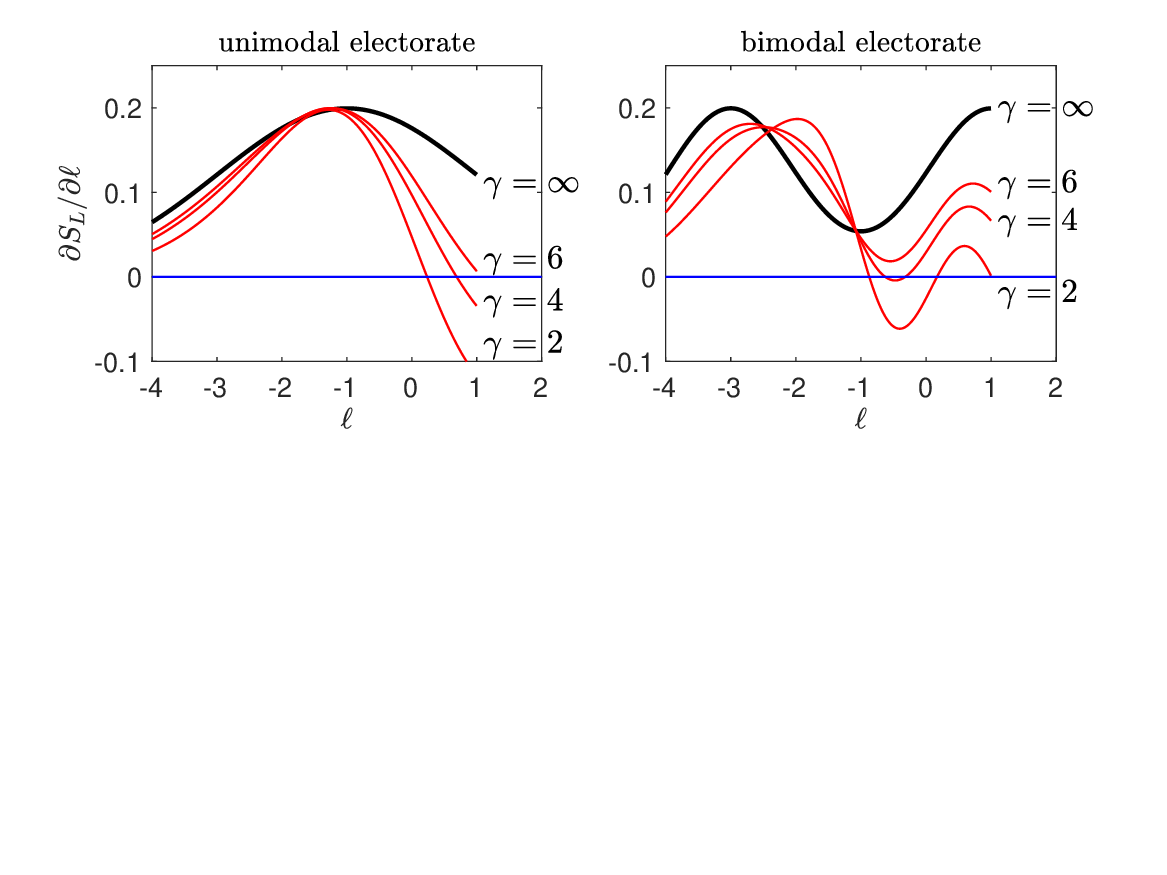}
\caption{$ \partial S_L/\partial \ell$ as a function of $\ell$, with $r=1$ fixed, and with the parameter $\gamma$  set to $\infty$ (black), 6, 4, and 2.}
\label{fig:SL_WITH_GAMMA_DERIVATIVES}
\ec
\end{figure}
\begin{figure}[ht]
\bc
\ig[scale=0.3]{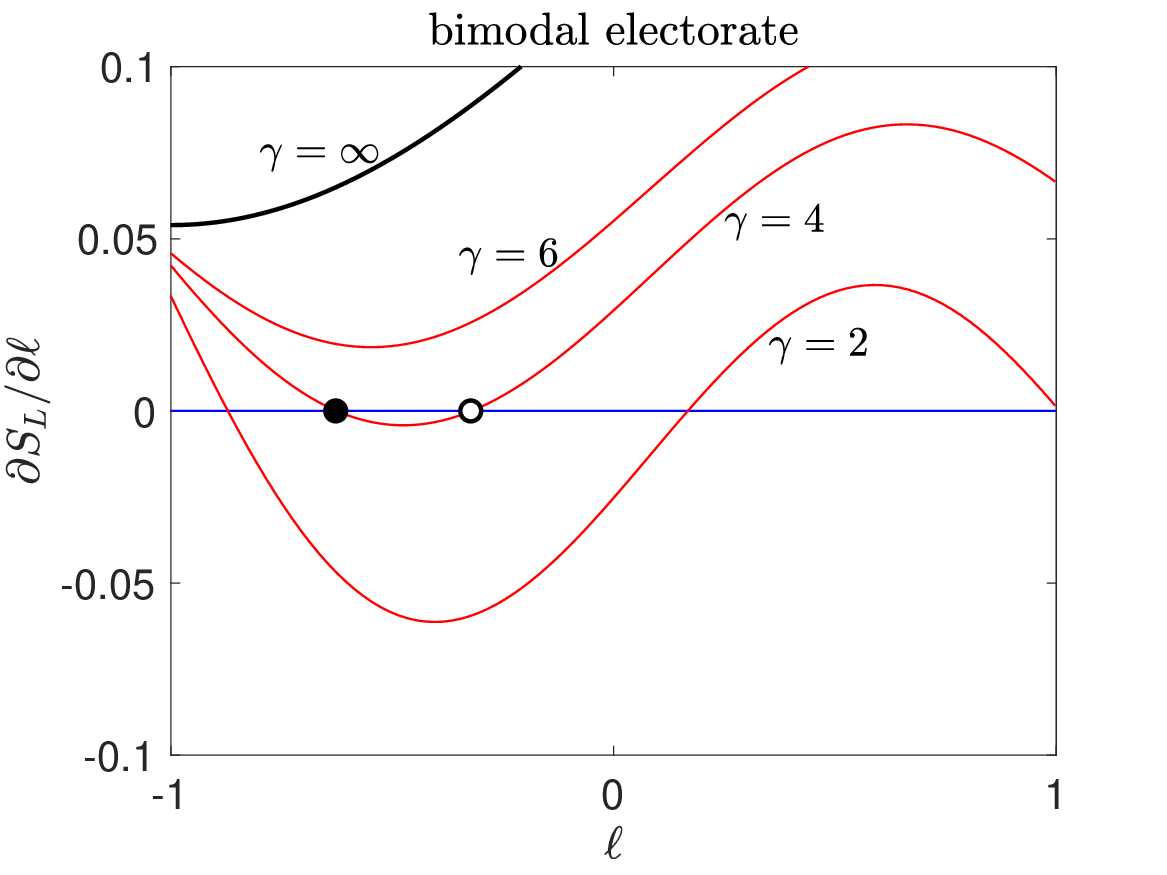}
\caption{Close-up of Figure \ref{fig:SL_WITH_GAMMA_DERIVATIVES}, right panel. For $\gamma=4$, the stable fixed point is indicated as a solid circle, and the 
unstable fixed point as an open circle.}
\label{fig:SL_WITH_GAMMA_DERIVATIVES_CLOSEUP}
\ec
\end{figure}

In Figure \ref{fig:DISCONTINUITY} we show, { for the bimodal electorate,}  $L$'s final position
$$
\ell_\infty = \min(\lim_{t \rightarrow \infty} \ell(t) ,1) 
$$
as a function of $\gamma$. 
If $\ell(t)$ were to rise above $1$, the assumptions under which our model was formulated would become invalid; we define
$\ell_\infty$ to be 1 in this case.
As $\gamma$ crosses $\gamma_c$, the value of $l_\infty$ jumps from a negative value, namely the location 
of the blue sky bifurcation on the $\ell$-axis, to $r=1$.
\begin{figure}[h!]
\bc
\ig[scale=0.35]{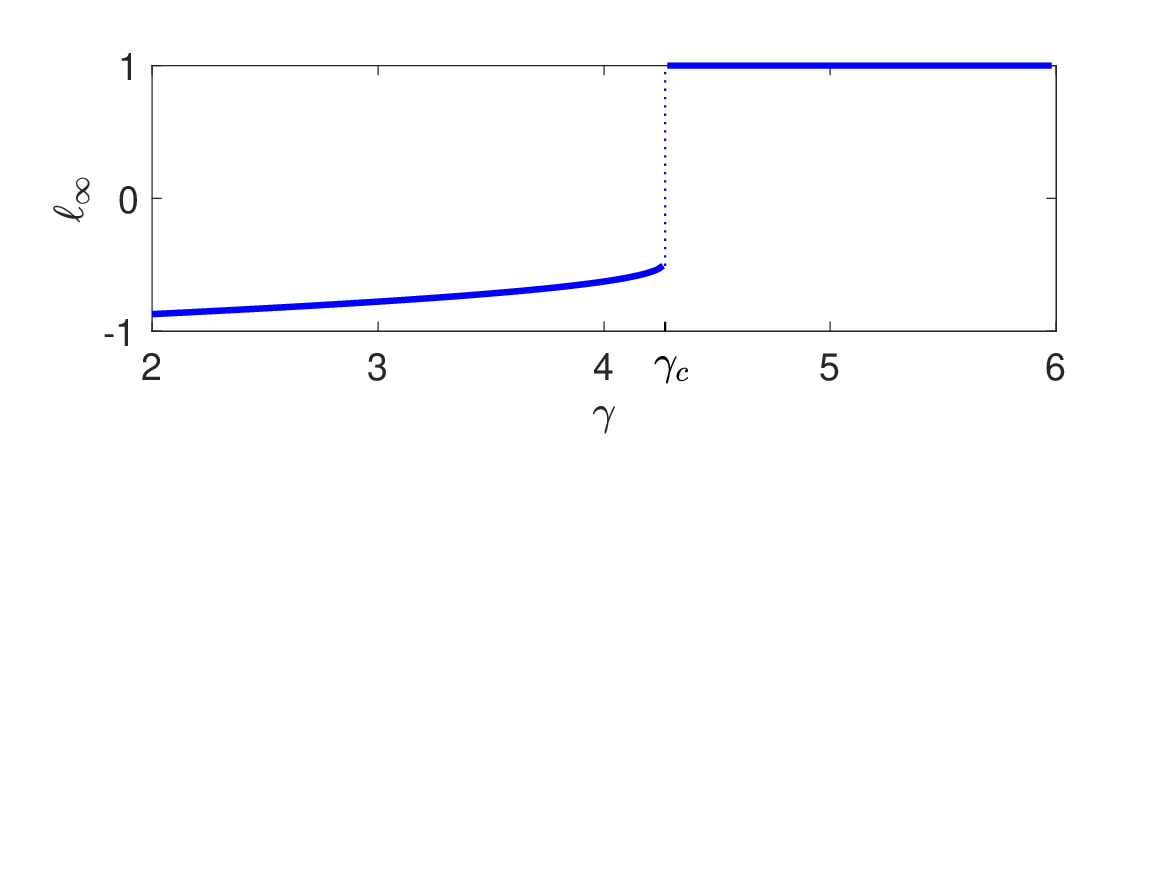}
\caption{$L$'s final position, assuming that $L$ starts out at $\ell=-1$ and follows the 
equation $\frac{d \ell}{dt} = \frac{\partial S_L}{\partial \ell}(\ell,1)$ { (so $\alpha=1$, $\beta=0$),} stopping when it reaches $\ell=1$.}
\label{fig:DISCONTINUITY} 
\ec
\end{figure}

\subsection{Two examples of two competing opportunists.} In Figure \ref{fig:TWO_ODES}, we show numerically computed solutions of 
equations (\ref{eq:ODE1}) and (\ref{eq:ODE2}) with $\alpha=1$, $\beta=0.1$, and $\gamma=3.78$ or $\gamma=3.8$.
As $\gamma$ rises from 3.78 to 3.8, there is an abrupt shift in $L$'s behavior: For $\gamma=3.78$, $L$ stays on the left. (Notice that $L$ appears
to consider moving to the center, then rejects it.) 
For $\gamma=3.8$, $L$ moves all the way to the right to meet $R$. The transition to coalescence is not always
 abrupt, as illustrated in our second example, Figure \ref{fig:TWO_ODES_2}.

\begin{figure}[ht]
\bc
\ig[scale=0.4]{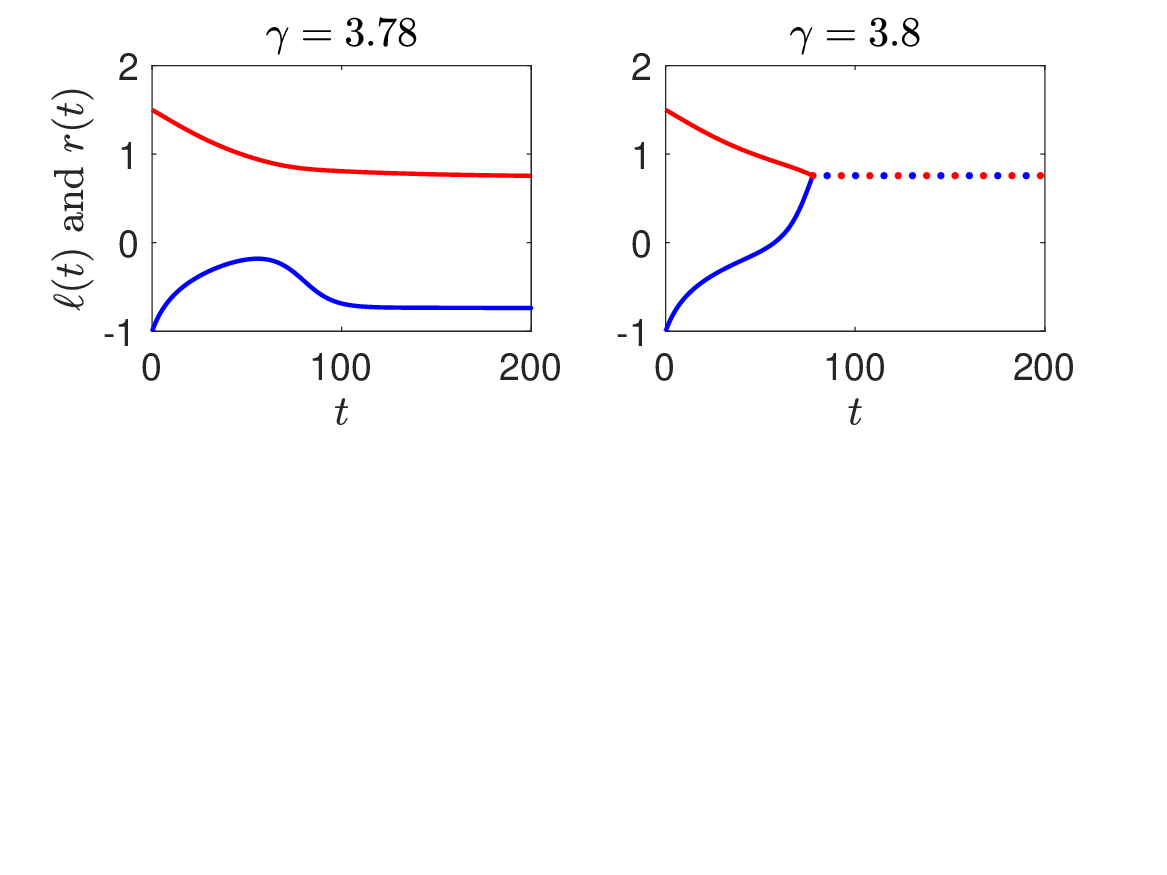}
\caption{Solutions of equations (\ref{eq:ODE1}) and (\ref{eq:ODE2}) with the bimodal $f$, and with $\alpha=1$, $\beta=0.1$, and $\gamma=3.78$ (left) 
and 3.8 (right). The very slight increase in $\gamma$ causes an abrupt shift in $L$'s final position.}
\label{fig:TWO_ODES}
\ec
\end{figure}
\begin{figure}[ht]
\bc
\ig[scale=0.4]{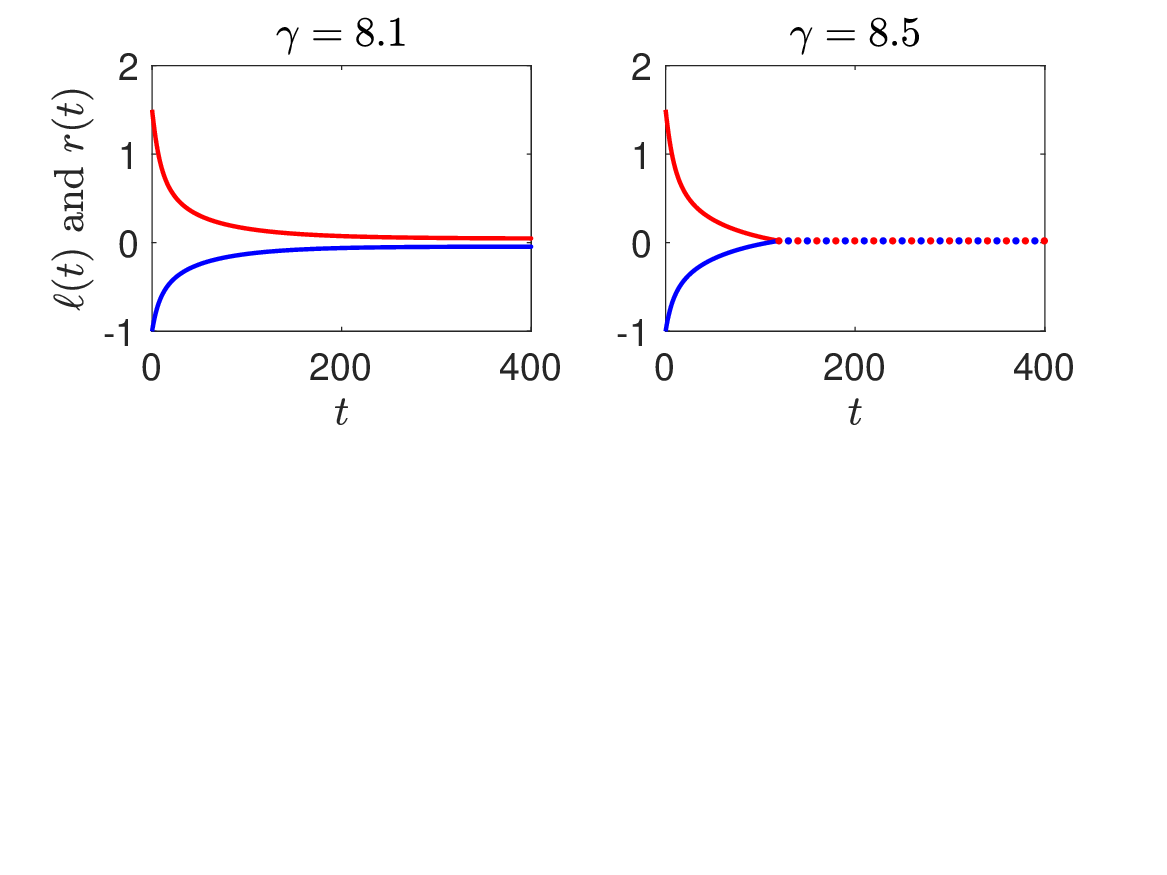}
\caption{Solutions of equations (\ref{eq:ODE1}) and (\ref{eq:ODE2}) with the bimodal $f$, and with $\alpha=1$, $\beta=1$, and $\gamma=8.1$ (left) 
and 8.5 (right).}
\label{fig:TWO_ODES_2}
\ec
\end{figure}

\subsection{Parameter space exploration.} We have  explained why the optimal position of $L$ depends discontinuously on $\gamma$
when $\beta=0$, and the voter density is bimodal. 
The larger question, of course, is how the appearance of discontinuities depends on the parameters $\alpha$ and $\beta$, and on
the voter density $f$. We have no general answer to this question. Our intent in this brief article is merely to show the {\em possibility} 
of discontinuous behavior, not to analyze broadly when it is seen.

We do, however, have pertinent numerical results. Let $\alpha>0$ and $\beta>0$, and let $\ell$ and $r$ denote the solutions of equations (\ref{eq:ODE1}) and (\ref{eq:ODE2}) with $\ell(0)=-1.5$ and $r(0)=1.5$. We make the convention that $\ell$ and $r$ freeze the moment they
come together. 
Define
$$
\ell_\infty = \lim_{t \rightarrow \infty} \ell(t), ~~~~ r_\infty=\lim_{t \rightarrow \infty} r(t), ~~~~q_\infty=r_\infty-\ell_\infty.
$$
Then $q_\infty$ is a function of $\alpha, \beta, \gamma$, and of course of the voter density $f$. 
However, scaling time amounts to scaling $\alpha$ and $\beta$ by the same factor.  Therefore 
$q_\infty$ is in fact, for a given $f$, only a function of $\beta/\alpha$ and $\gamma$. We therefore fix $\alpha=1$ and consider $q_\infty$ a function of $\beta$ and $\gamma$.

In Figure \ref{fig:Q_INFTY}, we show the graph 
of $q_\infty(\beta,\gamma)$ for the unimodal and 
bimodal electorates. For the unimodal electorate, $q_\infty$ is  zero for large $\gamma$, positive for
 small $\gamma$, and there is no discontinuity.
 For the bimodal electorate, 
a discontinuity appears, for an intermediate range of values of $\gamma$, and a range of values of $\beta$ between $0$ and approximately 0.8.
\begin{figure}[ht]
{ 
\bc
\ig[scale=0.6]{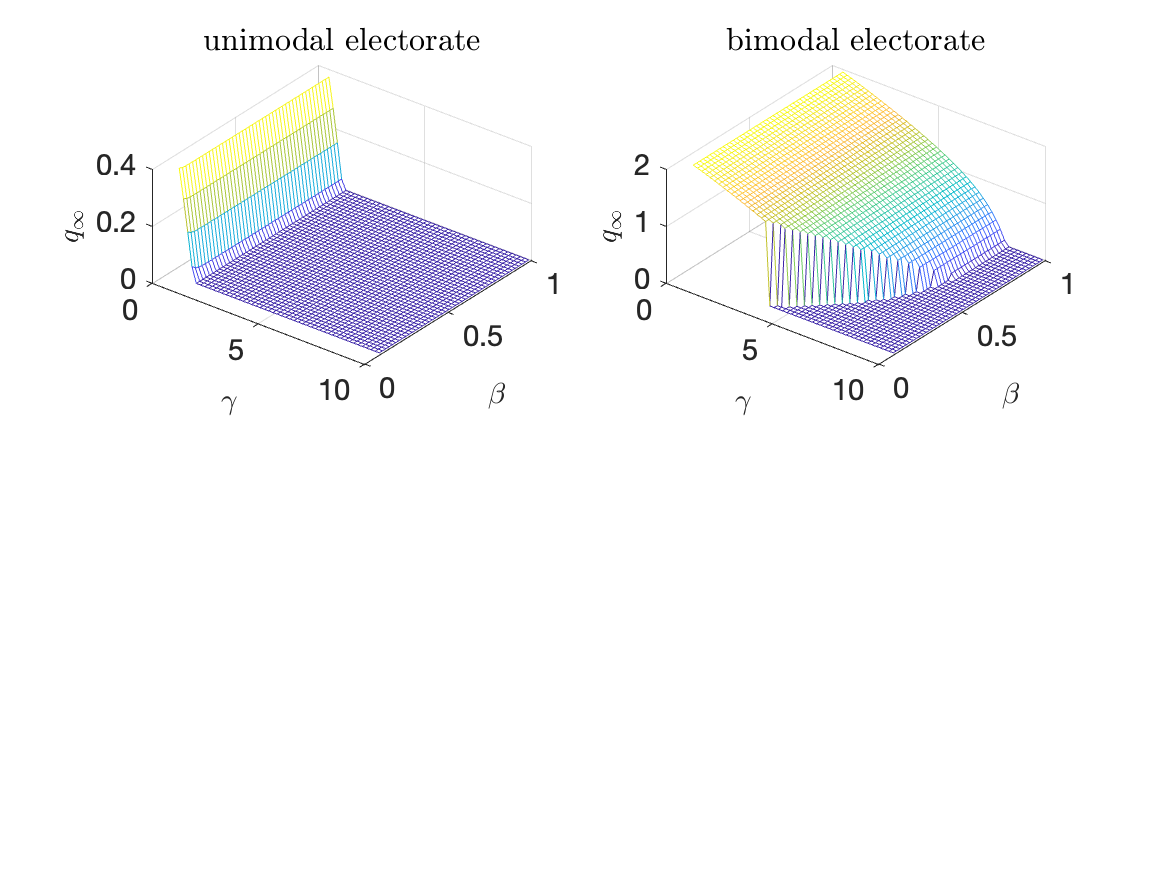}
\caption{ The distance between the final positions of $L$ and $R$, after completion of the optimization
process, as a function 
of $\beta$ ($\alpha$ is assumed to be 1) and $\gamma$.}
\label{fig:Q_INFTY}
\ec
}
\end{figure}
In fact, as long as $q_\infty>0$ (so the two candidates' positions never meet each other), $r_\infty$ and $l_\infty$, and therefore also $q_\infty$, 
are independent of $\alpha$ and $\beta$. They
are solutions of the system 
$$
\frac{\partial S_L}{\partial \ell}(\ell, r) = 0, ~~~~\frac{\partial S_R}{\partial r}(\ell, r) = 0, 
$$
which doesn't involve $\alpha$ and $\beta$; see equations (\ref{eq:ODE1}) and (\ref{eq:ODE2}). There is no  $\beta$-dependence in Fig.\ \ref{fig:Q_INFTY}, 
{\em except} at the discontinuity.

\section{Concluding remarks.}

\subsection{Related literature.} 
Ideas similar to ours have very recently been proposed by Jones {\em et al.} \cite{Jones_et_al_on_abstention}. { In detail, the modeling in 
\cite{Jones_et_al_on_abstention} differs from ours. For example, Jones {\em et al.} assume that voters are {\em less likely} to vote
for a candidate who is further away from them in political views than another candidate, but {\em may} still vote for that candidate. 
Our equations (\ref{eq:ODE1}) and (\ref{eq:ODE2}) reduce to  \cite[equation 8]{Jones_et_al_on_abstention} when $\alpha$ and $\beta$ are set to 1. It is precisely the difference
between $\alpha$ and $\beta$, however --- the difference in willingness to shift position opportunistically --- which makes
the discontinuities possible; see the right panel of Figure \ref{fig:Q_INFTY}.}

{The model of Siegenfeld and Bar-Yam \cite{siegenfeld_bar-yam} differs from ours as well, but the results bear striking resemblance. Siegenfeld and Bar-Yam study the parameter dependence of the opinion that best represents the electorate, meaning --- in our notation --- the $y$ that maximizes $\int_{-\infty}^\infty f(x) g(|y-x|) dx$. They find that $y$ can depend discontinuously on the  strengths of the ``left" and ``right" camps, assuming a bimodal electorate and
taking abstentions into consideration.

{ For other examples of discontinuities linked to 
blue sky bifurcations in models of human behavior, less 
closely related to our work, see  \cite{Couzin_et_al} and \cite{Tuzon_et_al}.
}

\subsection{Future work.} 

{ 
Several aspects of our model call for further analysis. A  complete bifurcation analysis
of equations (\ref{eq:ODE1}) and (\ref{eq:ODE2}) would be of interest. The parameters are $\alpha$, $\beta$, and $\gamma$; 
as noted earlier, $\alpha$ can be set to 1 without loss of generality. It would also be interesting to explore the dependence 
on $f$ and $g_0$.}

Our model paints an idealized picture of a far more complex reality. For instance, there is a distinction between 
ideology and party allegiance. Some Republicans in the United States hold liberal views, and some Democrats
hold conservative views \cite{Pew}. In addition, it is not  unheard of for a single person to hold
liberal views on some issues, and conservative ones on others --- the ``left-right" description is  a crude approximation
to reality.

{ This suggests moving to higher-dimensional opinion spaces. 
If we considered even a two-dimensional
opinion space, our differential equations model would become four-dimensional, 
and chaotic dynamics might become a possibility. We note that the McKelvey chaos theorem \cite{McKelvey_1976}, too, requires 
an opinion space of dimension $\geq 2$. Moving to higher
dimensions in a model like ours would be 
 another interesting direction for future work.
Higher-dimensional fully continuous opinion dynamics were studied for instance in 
 \cite{higher_dimensional_opinion_spaces}. We note, however, that the 
phrase ``continuous opinion dynamics" does not always mean what we take it to mean; for instance in 
\cite{Lorenz_2007}, the word ``continuous" in the title
means that opinions can take on a continuum of values, not 
that opinion space and time are continuous.}

In our model, the density $f$ is fixed. In reality, opinions in the electorate evolve, and a persuasive, charismatic, or demagogic candidate can contribute
to changes in $f$. The example of the 2016 and 2020 U.\ S.\ presidential elections
suggests that even the parameter $\gamma$, or more generally the function $g$, may depend on 
candidate positions; a  hard right-wing candidate might raise participation on the left, and vice versa.
In future work we will couple candidate dynamics with the opinion dynamics in the electorate.

\subsection{A concluding thought on recent U.S.\  presidential elections.} 

In 2016, Donald Trump may have won the Presidency by realizing that many voters on the right 
would abstain unless he adopted positions similar to theirs. The same strategy failed in 2020. It is tempting to
speculate that by then, passions were running so high and therefore abstentions had become so rare that  centrism
had again become the winning strategy.
The abruptness of the shift in the optimal candidate position may have caught many Americans by surprise, making
 them open to the suggestion that the 2020 election could not have
been conducted properly. 

\vskip -20pt
~
\begin{acknowledgment}{Acklnowledgments.} We thank the 
Data Intensive Studies Center 
at Tufts University for supporting this work with a seed grant,
as well as the three very thoughtful
anonymous reviewers and the 
{\em Monthly} editorial board, whose comments improved the paper greatly. 
\end{acknowledgment}


\begin{biog}
\item[Christoph B\"orgers]  is a Professor of Mathematics at Tufts University. His research interests include mathematical neuroscience,  numerical analysis, and more recently anomalous diffusion
and opinion dynamics.  In
2022 he was the recipient of Tufts University's Leibner Award for Excellence in Teaching and Advising.
\begin{affil}
Department of Mathematics, Tufts University, Medford, MA 02155 \\
christoph.borgers@tufts.edu
\end{affil}

\item[Bruce Boghosian] is a Professor of Mathematics at Tufts University, with secondary appointments in the Departments of Computer Science and Physics.  His research interests include mathematical fluid dynamics and kinetic theory, and, more recently, the application of these disciplines to problems of inequality and wealth distribution.  He has been a fellow of the American Physical Society since 2000, and a recipient of Tufts University’s Distinguished Scholar Award in 2010.

\begin{affil}
Department of Mathematics, Tufts University, Medford, MA 02155 \\
bruce.boghosian@tufts.edu
\end{affil}

\item[Natasa Dragovic]  is a Norbert Wiener Assistant Professor of Mathematics at Tufts University. Her research focuses on opinion dynamics, complex systems, and stochastic geometry. She holds a Ph.\ D.\  in Mathematics from the University of Texas at Austin.

\begin{affil}
Department of Mathematics, Tufts University, Medford, MA 02155 \\
natasa.dragovic@tufts.edu
\end{affil}

\item[Anna Haensch]  is a senior data scientist in the Tufts University Data Intensive Studies Center with a secondary appointment in the Department of Mathematics. She earned a Ph.\ D.\  in Mathematics from Wesleyan University. Her research lies at the intersection of mathematics and the social sciences, and it deals with the many ways that we can use data to make a safer and more equitable world. 

\begin{affil}
Data Intensive Studies Center, Tufts University, Medford, MA 02155 \\
anna.haensch@tufts.edu
\end{affil}

%
%
%
%
%

\end{biog}
\vfill\eject

  \end{document}